\documentclass[a4paper,final,oneside,notitlepage,11pt]{article}

\usepackage{amssymb}
\usepackage{amsmath}
\usepackage{amstext}
\usepackage[]{geometry}
\usepackage{mathrsfs}
\usepackage{euscript}
\usepackage[all]{xy}
\usepackage{xstring}
\usepackage{slashed}
\usepackage{mathtools}

\def\Z {\mathbb{Z}}

\def\R {\mathbb{R}}
\def\C {\mathbb{C}}

\def\id{\mathrm{id}}

\def\hc#1{\mathrm{h}_{#1}}
\def\h {\mathrm{H}}

\def\subset{\subseteq}

\renewcommand{\varepsilon}{\epsilon}
\renewcommand{\to}{\nobr\!\xymatrix@R=0cm@C=1.4em{\ar[r] &}\nobr}
\renewcommand{\mapsto}{\!\xymatrix@R=0cm@C=1.4em{\ar@{|->}[r] &}\!}
\renewcommand{\Rightarrow}{\!\xymatrix@R=0cm@C=1.4em{\ar@{=>}[r] &}\!}
\renewcommand{\Leftarrow}{\!\xymatrix@R=0cm@C=1.4em{\ar@{<=}[r] &}\!}
\newcommand{\incl}{\!\xymatrix@R=0cm@C=1.4em{\ar@{^(->}[r] &}\!}

\renewcommand\Leftrightarrow{\!\xymatrix@R=0cm@C=1.4em{\ar@{<=>}[r] &}\!}

\def\erf#1{(\ref{#1})}

\def\brackets#1{\IfStrEq{#1}{-}{}{(#1)}}
\def\subindex#1{\IfStrEq{#1}{-}{}{_{#1}}}

\newcommand{\alxydim}[2]{\begin{aligned}\xymatrix#1{#2}\end{aligned}}

\makeatletter
\def\bigset#1#2{\left\lbrace\;\begin{minipage}[c]{#1}\begin{center}#2\end{center}\end{minipage}\;\right\rbrace}

\makeatother

\newlength{\myl}

\def\ddt#1#2#3{\left.\frac{\mathrm{d}^{\IfStrEq{#1}{1}{}{#1}}}{\mathrm{d}#2}\IfStrEq{#2}{#3}{\right.}{\right|_{#3}}}

\def\nobr{~\hspace{-0.26em}}

\let\Oldin\in\renewcommand{\in}{\nobr\Oldin\nobr}
\let\Oldtimes\times\renewcommand{\times}{\nobr\Oldtimes}
\let\Oldotimes\otimes\renewcommand{\otimes}{\nobr\Oldotimes}

\newcommand{\ueins}{{\mathrm{U}}(1)}

\newcommand{\spin}[1]{{\mathrm{Spin}}\brackets{#1}}
\newcommand{\spinc}[1]{{{\mathrm{Spin}}^{{\C}}}\brackets{#1}}
\newcommand{\so}[1]{{\mathrm{SO}}\brackets{#1}}

\def\hom{\mathcal{H}\!om}
\def\homcon{\mathcal{H}\!om^{\!\nabla}\!}

\def\ev{\mathrm{ev}}

\newlength{\widthtmp}
\def\length#1{\settowidth{\widthtmp}{#1}\the\widthtmp}

\def\lli#1{\prescript{}{#1}}

\def\buntech#1#2{\mathcal{B}\hspace{-0.01em}un_{\hspace{0.05em}#1}^{#2}}

\def\buncon#1#2{\buntech{#1}{\nabla}\hspace{-0.05em}\brackets{#2}}

\def\fusbuncon#1#2{\mathcal{F}\!us\buncon{#1}{#2}}

\def\grbtech#1{\mathcal{G}\hspace{-0.06em}r\hspace{-0.06em}b_{\hspace{-0.07em}{#1}}}
\def\grb#1#2{\grbtech{#1}\brackets{#2}}
\def\grbcon#1#2{\grbtech{#1}^{\nabla\!}\brackets{#2}}

\def\exd#1{{#1^{\vee}}}

\def\ev{\mathrm{ev}}

\def\hc#1{\mathrm{h}_{#1}}
\def\pcomp{\nobr\star\nobr}

\def\prev#1{\overline{#1}}

\def\tr{\mathscr{T}}

\def\fusbuncon#1#2{\mathcal{F}\!us\buncon{#1}{#2}}

\def\ptr#1{\tau_{#1}}

\def\struc#1#2{#1\text{-}\mathcal{L}\!i\!f\!t(#2)}
\def\struccon#1#2#3{#1\text{-}\mathcal{L}\!i\!f\!t^{\nabla\!}_{#2}(#3)}

\usepackage{amsthm}
\usepackage{graphicx}
\usepackage[margin=2cm,font=normalsize,labelfont=bf]{caption}
\usepackage{verbatim}
\usepackage{enumerate}
\usepackage{fancyhdr}
\usepackage[]{geometry}
\usepackage[normalem]{ulem}
\usepackage{xstring}
\usepackage{needspace}
\usepackage{color}
\usepackage{amstext}

\makeatletter

\newcounter{denseversion}
\newcounter{authorcounter}
\newcounter{adresscounter}

\def\title#1{\gdef\@title{#1}}
\def\@title{}
\def\subtitle#1{\gdef\@subtitle{#1}}
\def\@subtitle{}

\def\authortagsused{0}
\def\adresstag#1{\if!#1!\else$^{\;#1\;}$\fi}

\renewcommand{\author}[2][]{
  \stepcounter{authorcounter}
  \if!#1!\else\gdef\authortagsused{1}\fi
  \ifnum\value{authorcounter}=1
    \def\@authorstringa{#2\adresstag{#1}}
    \def\@authorstringb{#2}
    \def\@authorstringc{#2\adresstag{#1}}
  \else
    \g@addto@macro\@authorstringa{\ and #2\adresstag{#1}}
    \g@addto@macro\@authorstringb{\ and #2}
    \g@addto@macro\@authorstringc{\\#2\adresstag{#1}}
  \fi}
\def\@author{\ifnum\value{denseversion}=0\@authorstringa\else\@authorstringb\fi}

\def\@adressstringa{}
\def\@adressstringb{}
\newcommand{\adress}[2][]{
  \stepcounter{adresscounter}
  \ifnum\value{adresscounter}=1
    \g@addto@macro\@adressstringa{\ifnum\authortagsused=0\def\br{\\}\else\def\br{, }\fi\adresstag{#1}#2}
    \g@addto@macro\@adressstringb{\def\br{\\}\adresstag{#1}\parbox[t]{14cm}{#2}}
  \else
    \g@addto@macro\@adressstringa{\\[\bigskipamount]\adresstag{#1}#2}
    \g@addto@macro\@adressstringb{\\[\medskipamount]\adresstag{#1}\parbox[t]{14cm}{#2}}
  \fi}

\def\preprint#1{\gdef\@preprint{#1}}
\def\@preprint{}
\def\keywords#1{\gdef\@keywords{#1}}
\def\@keywords{}
\def\msc#1{\gdef\@msc{#1}}
\def\@msc{}
\def\email#1{
   \gdef\@email{#1}
   \g@addto@macro\@authorstringc{ {\it (#1)}}}
\def\@email{}
\def\dedication#1{\gdef\@dedication{#1}}
\def\@dedication{}

\def\mybaselinestretch#1{\gdef\@mybaselinestretch{#1}}
\def\@mybaselinestretch{}

\def\refname{References}

\mybaselinestretch{1.2}
\renewcommand{\baselinestretch}{\@mybaselinestretch}

\def\denseversion{
  \setcounter{denseversion}{1}
  \newgeometry{left=3cm,right=3cm,top=3cm}
  \mybaselinestretch{1.1}
  \renewcommand{\baselinestretch}{\@mybaselinestretch}
  \normalfont
  \def\possiblelinebreak{}
  \fancyfoot[C]{\itshape{--$\,\,$\thepage$\,\,$--}}}

\newlength{\myparskip}
\newlength{\myproofparskip}

\def\possiblelinebreak{\\}

\renewcommand{\emph}[1]{\def\reserved@a{it}\ifx\f@shape\reserved@a\uline{#1}\else\textit{#1}\fi}

\newcommand{\mytableofcontents}{
   \ifnum\value{denseversion}=0
     \renewcommand{\baselinestretch}{1}
     \normalfont
     \tableofcontents
     \renewcommand{\baselinestretch}{\@mybaselinestretch}
     \normalfont
   \else
     \renewcommand{\baselinestretch}{0.5}
     \normalfont
     \tableofcontents
     \renewcommand{\baselinestretch}{\@mybaselinestretch}
     \normalfont
   \fi}

\newlength{\zeilenlaenge}
\def\putindent#1{
  \settowidth{\zeilenlaenge}{#1}
  \ifnum\zeilenlaenge>\textwidth
    #1
  \else
    \noindent #1
  \fi
}

\def\href#1#2{#2}

\def\kohyp{
  \usepackage{hyperref}
  \hypersetup{
    linktocpage = true,
    pdftitle = {\@title},
    pdfauthor = {\@author},
    pdfkeywords = {\@keywords},    
    bookmarksopen = true,
    bookmarksopenlevel = 1
  }}  
\def\showkeywords{\begin{flushleft}\footnotesize\textbf{Keywords}: \@keywords\end{flushleft}}
\def\showmsc{\begin{flushleft}\footnotesize\textbf{MSC 2010}: \@msc\end{flushleft}}

\newcounter{mythm}[subsection]
\newcounter{mainthm}

\def\setsecnumdepth#1{
  \setcounter{secnumdepth}{#1}
  \setcounter{mythm}{0}
  \ifnum \c@secnumdepth >0
    \ifnum \c@secnumdepth >1
      \def\themythm{\thesubsection.\arabic{mythm}}
      \numberwithin{equation}{subsection}
      \renewcommand\theequation{\thesubsection.\arabic{equation}}
    \else
      \def\themythm{\thesection.\arabic{mythm}}
      \numberwithin{equation}{section}
      \renewcommand\theequation{\thesection.\arabic{equation}}
    \fi
  \else
    \def\themythm{\arabic{mythm}}
  \fi}

\newenvironment{mythmenv}{\strut\ \setlength{\parskip}{\myproofparskip}}{\setlength{\parskip}{\myparskip}}

\newlength{\mythmskip}
\newlength{\mythmtopskip}
\newtheoremstyle{mythmstylea}{\mythmtopskip}{\mythmskip}{\it}{}{\bf}{.}{0em}{}
\newtheoremstyle{mythmstyleb}{\mythmtopskip}{\mythmskip}{}{}{\bf}{.}{0em}{}

\theoremstyle{mythmstylea}
\newtheorem{mytheorem}[mythm]{\nameTheorem}
\newtheorem{mydefinition}[mythm]{\nameDefinition}
\newtheorem{mycorollary}[mythm]{\nameCorollary}
\newtheorem{myproposition}[mythm]{\nameProposition}
\newtheorem{mylemma}[mythm]{\nameLemma}

\newtheorem{mymaintheorem}[mainthm]{\nameTheorem}
\newtheorem{mymaincorollary}[mainthm]{\nameCorollary}
\newtheorem{mymainproposition}[mainthm]{\nameProposition}
\newtheorem{mymaindefinition}[mainthm]{\nameDefinition}

\theoremstyle{mythmstyleb}

\newtheorem{myremark}[mythm]{\nameRemark}
\newtheorem{myproblem}[mythm]{\nameProblem}
\newtheorem{myexample}[mythm]{\nameExample}
\newtheorem{myexercise}[mythm]{\nameExercise}

\newenvironment{theorem}[1][]{\begin{mytheorem}[#1]\begin{mythmenv}}{\end{mythmenv}\end{mytheorem}}
\newenvironment{definition}[1][]{\begin{mydefinition}[#1]\begin{mythmenv}}{\end{mythmenv}\end{mydefinition}}
\newenvironment{corollary}[1][]{\begin{mycorollary}[#1]\begin{mythmenv}}{\end{mythmenv}\end{mycorollary}}

\newenvironment{lemma}[1][]{\begin{mylemma}[#1]\begin{mythmenv}}{\end{mythmenv}\end{mylemma}}
\newenvironment{remark}[1][]{\begin{myremark}[#1]\begin{mythmenv}}{\end{mythmenv}\end{myremark}}

\newenvironment{maintheorem}[1]{\def\themainthm{#1}\begin{mymaintheorem}\begin{mythmenv}}{\end{mythmenv}\end{mymaintheorem}}

\renewenvironment{proof}[1][\nameProof]{\noindent #1. \begin{mythmenv}}{\hphantom{$\square$}\hfill$\square$\end{mythmenv}\medskip}

\def\mytitle{}
\def\zmptitle{
  \begin{tabular}{cc}
    \begin{minipage}[c]{0.4\textwidth}
      \begin{flushleft}
        \includegraphics[width=110pt]{../../tex/zmp}
      \end{flushleft}  
    \end{minipage}&
    \begin{minipage}[c]{0.55\textwidth}
      \begin{flushright}
      {\small\sf\@preprint}
      \end{flushright}
    \end{minipage}
  \end{tabular}
  \vskip 2cm}

\def\maketitle{
  \setlength{\parskip}{\myparskip}  
  \newpage
  \noindent
  \mytitle
  \begin{center}
    \LARGE\@title\\
    \if!\@subtitle!\else\smallskip\LARGE\@subtitle\\\fi
    \bigskip
    \if!\@author!\else\bigskip\large\@author\\\fi
    \ifnum\value{denseversion}=0
      \if!\@adressstringa!\else\bigskip\normalsize\@adressstringa\\\fi
      \if!\@email!\else\ifnum\value{authorcounter}=1\bigskip\normalsize\textit{\@email}\\\else\fi\fi
    \else
    \fi
    \if!\@dedication!\else\bigskip\normalsize{\@dedication}\\\fi
  \end{center}
  \ifnum\value{denseversion}=0\vskip 1.5cm\else\vskip0.5cm\fi
  \if!\@draft!\else\thispagestyle{empty}\fi}

\def\kobiburl#1{
   \IfBeginWith
     {#1}
     {http://arxiv.org/abs/}
     {\kobibarxiv{#1}}
     {\kobiblink{#1}}}
\def\kobibarxiv#1{\href{#1}{\texttt{[arxiv:\StrGobbleLeft{#1}{21}]}}}
\def\kobiblink#1{
  \StrSubstitute{#1}{_}{\underline{\;\;}}[\mylink]
  \StrSubstitute{\mylink}{&}{\&}[\mylink]
  \StrSubstitute{\mylink}{/}{/\allowbreak}[\mylink]
  \newline Available as: \mbox{\;}
  \href{#1}{\texttt{\mylink}}}

\def\nolinks{\def\kobiblink##1{}}

\def\kobib#1{
  \begin{raggedright}
  \ifnum\value{denseversion}=0\else\small\fi
  \Oldbibliography{#1/kobib}
  
\begin{thebibliography}{CJM02}
\addcontentsline{toc}{section}{\refname}

\bibitem[Ati85]{atiyah2}
M.~F. Atiyah, \quot{Circular symmetry and stationary phase approximation}.
\newblock {\em Ast\'erisque}, 131:43--59, 1985.
\bibitem[BM94]{brylinski4}
J.-L. Brylinski and D.~A. McLaughlin, \quot{The geometry of degree four
  characteristic classes and of line bundles on loop spaces {I}}.
\newblock {\em Duke Math. J.}, 75(3):603--638, 1994.
\bibitem[BM96]{brylinski5}
J.-L. Brylinski and D.~A. McLaughlin, \quot{The geometry of degree four
  characteristic classes and of line bundles on loop spaces {II}}.
\newblock {\em Duke Math. J.}, 83(1):105--139, 1996.
\bibitem[Bry93]{brylinski1}
J.-L. Brylinski, {\em Loop spaces, characteristic classes and geometric
  quantization}, volume 107 of {\em Progr. Math.}
\newblock Birkh\"auser, 1993.
\bibitem[CJM02]{carey2}
A.~L. Carey, S.~Johnson, and M.~K. Murray, \quot{Holonomy on {D}-branes}.
\newblock {\em J. Geom. Phys.}, 52(2):186--216, 2002.
\newblock \kobiburl{http://arxiv.org/abs/hep-th/0204199}
\bibitem[CW08]{carey6}
A.~L. Carey and B.-L. Wang, \quot{Fusion of symmetric D-Branes and Verlinde
  Rings}.
\newblock {\em Commun. Math. Phys.}, 277(3):577--625, 2008.
\newblock \kobiburl{http://arxiv.org/abs/math-ph/0505040}
\bibitem[Gom03]{gomi3}
K.~Gomi, \quot{Connections and curvings on lifting bundle gerbes}.
\newblock {\em J. Lond. Math. Soc.}, 67(2):510--526, 2003.
\newblock \kobiburl{http://arxiv.org/abs/math/0107175}
\bibitem[McL92]{mclaughlin1}
D.~A. McLaughlin, \quot{Orientation and string structures on loop space}.
\newblock {\em Pacific J. Math.}, 155(1):143--156, 1992.
\bibitem[Mur96]{murray}
M.~K. Murray, \quot{Bundle gerbes}.
\newblock {\em J. Lond. Math. Soc.}, 54:403--416, 1996.
\newblock \kobiburl{http://arxiv.org/abs/dg-ga/9407015}
\bibitem[Mur10]{murray3}
M.~K. Murray, \quot{An introduction to bundle gerbes}.
\newblock In O.~Garcia-Prada, J.~P. Bourguignon, and S.~Salamon, editors, {\em
  The many facets of geometry. A tribute to {N}igel {H}itchin}. Oxford Univ.
  Press, 2010.
\newblock \kobiburl{http://arxiv.org/abs/0712.1651}
\bibitem[ST]{stolz3}
S.~Stolz and P.~Teichner, \quot{The spinor bundle on loop spaces}.
\newblock Preprint.
\newblock
  \kobiburl{http://people.mpim-bonn.mpg.de/teichner/Math/Surveys_files/MPI.pdf}
\bibitem[Ste00]{stevenson1}
D.~Stevenson, {\em The geometry of bundle gerbes}.
\newblock PhD thesis, University of Adelaide, 2000.
\newblock \kobiburl{http://arxiv.org/abs/math.DG/0004117}
\bibitem[SW07]{spera1}
M.~Spera and T.~Wurzbacher, \quot{Twistor spaces and spinors over loop spaces}.
\newblock {\em Math. Ann.}, 338:801--843, 2007.
\bibitem[SW09]{schreiber3}
U.~Schreiber and K.~Waldorf, \quot{Parallel transport and functors}.
\newblock {\em J. Homotopy Relat. Struct.}, 4:187--244, 2009.
\newblock \kobiburl{http://arxiv.org/abs/0705.0452v2}
\bibitem[SW10]{schweigert2}
C.~Schweigert and K.~Waldorf, \quot{Gerbes and {L}ie groups}.
\newblock In K.-H. Neeb and A.~Pianzola, editors, {\em Developments and trends
  in infinite-dimensional {L}ie theory}, volume 600 of {\em Progr. Math.}
  Birkh\"auser, 2010.
\newblock \kobiburl{http://arxiv.org/abs/0710.5467}
\bibitem[Wal]{waldorf10}
K.~Waldorf, \quot{Transgression to loop spaces and its inverse, {II}: Gerbes
  and fusion bundles with connection}.
\newblock {\em Asian J. Math.}, to appear.
\newblock \kobiburl{http://arxiv.org/abs/1004.0031}
\bibitem[Wal07]{waldorf1}
K.~Waldorf, \quot{More morphisms between bundle gerbes}.
\newblock {\em Theory Appl. Categ.}, 18(9):240--273, 2007.
\newblock \kobiburl{http://arxiv.org/abs/math.CT/0702652}
\bibitem[Wal10]{waldorf5}
K.~Waldorf, \quot{Multiplicative bundle gerbes with connection}.
\newblock {\em Differential Geom. Appl.}, 28(3):313--340, 2010.
\newblock \kobiburl{http://arxiv.org/abs/0804.4835v4}
\bibitem[Wal12]{waldorf9}
K.~Waldorf, \quot{Transgression to loop spaces and its inverse, {I}:
  Diffeological bundles and fusion maps}.
\newblock {\em Cah. Topol. G\'eom. Diff\'er. Cat\'eg.}, LIII:162--210, 2012.
\newblock \kobiburl{http://arxiv.org/abs/0911.3212}
\end{thebibliography}

  \end{raggedright}
  \ifnum\value{denseversion}=0\else
      \noindent
      \if!\@authorstringc!\else
        \ifnum\authortagsused=0\ifnum\value{authorcounter}>1\normalsize\@authorstringc\\[\medskipamount]\else\fi\else\normalsize\@authorstringc\\[\medskipamount]\fi
      \fi
      \if!\@adressstringb!\else\normalsize\@adressstringb\\{}\fi
      \ifnum\authortagsused=0
        \ifnum\value{authorcounter}=1
          \if!\@email!\else\linebreak\normalsize\textit{\@email}\\{}\fi
        \else
        \fi
      \else
      \fi
  \fi
  }

\let\Oldbibliography

\def\bibliography#1{
  \begin{raggedright}
  \ifnum\value{denseversion}=0\else\small\fi
  \Oldbibliography{#1}
  \end{raggedright}
  \ifnum\value{denseversion}=0\else
      \noindent
      \if!\@authorstringc!\else
        \ifnum\authortagsused=0\ifnum\value{authorcounter}>1\normalsize\@authorstringc\\[\medskipamount]\else\fi\else\normalsize\@authorstringc\\[\medskipamount]\fi
      \fi
      \if!\@adressstringb!\else\normalsize\@adressstringb\\{}\fi
      \ifnum\authortagsused=0
        \ifnum\value{authorcounter}=1
          \if!\@email!\else\linebreak\normalsize\textit{\@email}\\{}\fi
        \else
        \fi
      \else
      \fi
  \fi
}

\newenvironment{commentfigure}{\begin{comment}}{\end{comment}}
\newenvironment{sidewayscommentfigure}{\begin{minipage}}{\end{minipage}}

\def\draft#1#2#3#4{
  \ifnum#4=0
    \def\showcomments{ - Comments are not displayed}
  \else
    \renewenvironment{comment}{\begin{list}{}{\rightmargin=1cm\leftmargin=1cm}\item\sf\begin{small}}{\end{small}\end{list}}

    \def\showcomments{ - Comments are displayed}
  \fi
  \gdef\@draft{DRAFT - Version #1 - Last edited on #2 - Last edited by #3\showcomments}
  \fancyhead[C]{\footnotesize\tt\textcolor{red}{\@draft}}}
\def\@draft{}

\makeatother

\usepackage[latin1]{inputenc}
\usepackage[english]{babel}

\def\quot#1{``#1''}

\def\quand{\quad\text{ and }\quad}
\def\quomma{\quad\text{, }\quad}

\def\quith{\quad\text{ with }\quad}

\def\nameTheorem{Theorem}
\def\nameDefinition{Definition}
\def\nameCorollary{Corollary}
\def\nameProposition{Proposition}
\def\nameLemma{Lemma}
\def\nameRemark{Remark}
\def\nameProblem{Problem}
\def\nameExample{Example}
\def\nameExercise{Exercise}
\def\nameProof{Proof}

\title{A Loop Space Formulation for Geometric Lifting Problems}
\author{Konrad Waldorf}
\adress{Department of Mathematics\\University of California, Berkeley\\970 Evans Hall \#3840\\Berkeley, CA 94720, USA}
\email{waldorf@math.berkeley.edu}
\dedication{Dedicated to Alan Carey, on the occasion of his 60th birthday}

\kohyp
\usepackage{amstext}

\begin{document}


\def\hat#1{\widehat{#1}}
\def\fussec#1#2{\Gamma^{f\!u\!s}_{#1}(#2)}

\maketitle

\begin{abstract}
We review and then combine two aspects of the theory of bundle gerbes.
The first concerns  lifting bundle gerbes and connections on those, developed by Murray and Gomi. Lifting gerbes represent obstructions against extending the structure group of a principal bundle. The second  is the transgression of gerbes to loop spaces, initiated by Brylinski and McLaughlin and with recent contributions of the author. Combining  these two aspects, we obtain a new formulation of  lifting problems  in terms of geometry on the loop space. Most prominently, our formulation explains the relation between (complex) spin structures on a Riemannian manifold and  orientations of its loop space.
\end{abstract}


\nolinks

\setsecnumdepth{1}

\section{Introduction and Statement of the Result}

In their seminal work \cite{brylinski4,brylinski5} on the geometry of line bundles over loop spaces, Brylinski and McLaughlin encounter an interesting \quot{product} such line bundles can be endowed with --- we are going to call it \emph{fusion product}.  The idea of a  fusion product on a line bundle $\mathscr{L}$ over the loop space is that it provides for two loops $\tau_1$ and $\tau_2$  that are smoothly composable to a third loop $\tau_2 \pcomp \tau_1$
a linear isomorphism
\begin{equation*}
\lambda: \mathscr{L}_{\tau_1} \otimes \mathscr{L}_{\tau_2} \to \mathscr{L}_{\tau_2 \pcomp \tau_1}\text{,}
\end{equation*}  
where $\mathscr{L}_{\tau}$ denotes the fibre of $\mathscr{L}$ over a loop $\tau$. 
A complete and slightly modified definition will be given later (Definition \ref{def:fusionproduct}).

\def\adjust#1{\!\!\!\begin{tabular}{c}$#1$\end{tabular}\!\!\!}

In this note we construct examples of  bundles with fusion products  in the context of geometric lifting problems. 
A \emph{lifting problem} is posed by specifying  a central extension
\begin{equation*}
\alxydim{}{1 \ar[r] & A \ar[r] & \adjust{\hat G}  \ar[r] & G \ar[r] & 1}
\end{equation*} 
of Lie groups and a principal $G$-bundle $P$ over a smooth manifold $M$. A solution for the lifting problem is a principal $\hat G$-bundle $\hat P$ over $M$ together with an equivariant bundle map from $\hat P$ to $P$, in the following called  \emph{$\hat G$-lift of $P$}. 
A \emph{geometric} lifting problem is one where  $P$ carries a connection, and a \emph{geometric $\hat G$-lift} $\hat P$ includes a connection on $\hat P$ that is compatible with the given one in a certain way. As we are going to explain later in more detail, geometric $\hat G$-lifts have a \emph{scalar curvature}: a 2-form on $M$ with values in the Lie algebra $\mathfrak{a}$ of the central Lie group $A$.

We establish a relation between geometric lifting problems and the geometry of bundles over the  loop space $LM := C^{\infty}(S^1,M)$ by constructing a principal $A$-bundle $\mathscr{L}_P$ over $LM$ from  a geometric lifting problem posed by a  principal $G$-bundle $P$ with connection over $M$.
Basically, the fibre of $\mathscr{L}_P$ over a loop $\tau\in LM$ consists of all geometric  $\hat G$-lifts of the pullback $\tau^{*}P$. We will see that every \quot{global} geometric $\hat G$-lift $\hat P$ defines -- by restricting it to  loops -- a smooth section $\sigma_{\hat P}: LM \to \mathscr{L}_P$.  Sections that can be obtained in this way turn out to be very particular.

The bundle $\mathscr{L}_P$ fits well into the context of the work of Brylinski and McLaughlin: it  comes with a fusion product and with a connection. We will see that the section $\sigma_{\hat P}$ is compatible with this additional structure. Firstly, it \emph{preserves} the fusion product. Secondly, its \emph{curvature} (i.e. the pullback of the connection of $\mathscr{L}_P$ to $LM$) coincides with the transgression of the scalar curvature of $\hat P$ (i.e. the pullback along the evaluation map $\ev: S^1 \times LM \to M$, followed by integration over the fibre $S^1$). We show that these two properties characterize those sections of $\mathscr{L}_P$ that come from geometric $\hat G$-lifts of $P$, and so establish the following loop space formulation for geometric lifting problems.

\begin{maintheorem}{A}
\label{main}
Let $M$ be a connected smooth manifold and $P$ be a principal $G$-bundle with connection over $M$, let $\rho \in \Omega^2(M,\mathfrak{a})$ and  $L\rho \in \Omega^1(LM,\mathfrak{a})$ denote its transgression. Then, the assignment $\hat P \mapsto \sigma_{\hat P}$ defines a bijection
\begin{equation*}
\bigset{3.9cm}{Equivalence classes of geometric $\hat G$-lifts of $P$ with scalar curvature $\rho$} \cong 
\bigset{4.5cm}{Fusion-preserving sections of $\mathscr{L}_P$ with curvature $-L\rho$}\text{.}
\end{equation*}
\end{maintheorem}

The definition of the bundle $\mathscr{L}_P$, the properties of the sections $\sigma_{\hat P}$, and the proof of Theorem \ref{main} are all obtained using the theory of \emph{bundle gerbes}. In Section \ref{sec:lifting} of this note  we review  lifting gerbes and connections on those following Murray and Gomi \cite{murray,gomi3}, respectively. The main result of Section \ref{sec:lifting}, Theorem \ref{th:lifting}, gives a complete formulation of geometric lifting problems in terms of bundle gerbes. In Section \ref{sec:transgression} we review the transgression of bundle gerbes to loop spaces, developed by Brylinski and McLaughlin \cite{brylinski4,brylinski5,brylinski1}, and include some recent contributions of the author \cite{waldorf10}. The main result of Section \ref{sec:transgression}, Theorem \ref{th:transgression}, is an equivalence between the category of bundle gerbes with connection over $M$ and a category of principal bundles with fusion products and connections over $LM$. In Section \ref{sec:proof} we put the two pieces together: we define the bundle $\mathscr{L}_P$ to be the transgression of a lifting gerbe, and combine the  equivalences of Theorems \ref{th:lifting} and \ref{th:transgression} to a one-line-proof of Theorem \ref{main}.

The remaining two sections of this note are complementary. In Section \ref{sec:alternative} we provide an alternative construction of the bundle $\mathscr{L}_P$ which is elementary in the sense that it does not use any gerbe theory (Theorem \ref{prop:transgression}). We will also construct the fusion product in that context --- so far as I know this is  the first elementary non-trivial example of a fusion product. In Section \ref{sec:spin} we apply Theorem \ref{main} to spin structures and complex spin structures on Riemannian manifolds. This is interesting because the bundle $\mathscr{L}_P$ has a nice interpretation as the \emph{orientation bundle} of $LM$.  Theorem \ref{main} reduces in the spin case to previously known results of Atiyah \cite{atiyah2} and Stolz-Teichner \cite{stolz3} (Corollary \ref{co:spin}), while in the complex spin case it provides a new description of complex spin structures in terms of certain loop space orientations (Corollary \ref{co:spinc}).

One line for further research could be to understand the role of fusion products in the case that the underlying manifold is a Lie group. There, it  can be seen as an additional structure for loop group extensions, and it is expected that it is responsible for the fusion of positive energy representations, thus the terminology. It would be interesting to have a geometrical formulation of  fusion in the setting of bundle gerbes and bundle gerbe modules, as initiated in \cite{carey6}.

\paragraph{Acknowledgements.} 
I gratefully acknowledge  a Feodor-Lynen scholarship, granted by the Alex\-an\-der von Hum\-boldt Foundation. I would also like to thank  the organizers of Alan Carey's 60th birthday conference for  inviting me to talk and to write this contribution to its proceedings.

\section{Lifting Bundle Gerbes}

\label{sec:lifting}

In this section we review (and slightly complete) the theory of lifting bundle gerbes and connections on them. The setup is a central extension
\begin{equation}
\label{ce}
\alxydim{}{1 \ar[r] & A \ar[r] & \adjust{\hat G} \ar[r]^{t} & G \ar[r] & 1}
\end{equation} 
of Lie groups, and a principal $G$-bundle $P$ over a smooth manifold $M$. A \emph{$\hat G$-lift of $P$} is a principal $\hat G$-bundle $\hat P$ over $M$ together with a bundle map $f: \hat P \to P$ satisfying $f(\hat p \cdot\hat g) = f(\hat p) \cdot t(\hat g)$ for all $\hat p \in \hat P$ and $\hat g\in \hat G$. $\hat G$-lifts of $P$ form a category $\struc {\hat G}P$. The existence of $\hat G$-lifts is obstructed by a class  $\xi_P \in \h^2(M,\underline{A})$ that is obtained by locally lifting a \v Cech cocycle for $P$ and then measuring the error.

The idea of realizing the obstruction class $\xi_P$ geometrically has been proposed by Brylinski \cite{brylinski1} in terms of Dixmier-Douady sheaves of groupoids. Murray has  adapted this idea to bundle gerbes, where it becomes particularly elegant \cite{murray}. I will assume in the following that the reader is a bit familiar with bundle gerbes --- for instance, the papers \cite{carey2,schweigert2,murray3} contain introductions.

Associated to the given bundle $P$ is the following bundle gerbe $\mathcal{G}_P$ over $M$, called the \emph{lifting gerbe}. Its surjective submersion is the bundle projection $\pi: P \to M$. We are going to denote its $k$-fold fibre product by $P^{[k]}$, and by $\pi_{i_1...i_k}: P^{[j]} \to P^{[k]}$ the projections to the indexed factors. Over $P^{[2]}$, the lifting gerbe has the  principal $A$-bundle $Q := g^{*}\hat G$, obtained by regarding $\hat G$ as a principal $A$-bundle over $G$, and pulling it back along the map $g:P^{[2]} \to G$  defined by $p\cdot g(p,p') = p'$. Finally, the multiplication of $\hat G$ defines a bundle gerbe product, i.e. a bundle isomorphism
\begin{equation}
\label{bgprod}
\mu: \pi_{12}^{*}Q \otimes \pi_{23}^{*}Q \to \pi_{13}^{*}Q
\end{equation}
over $P^{[3]}$ that is associative over $P^{[4]}$. The characteristic class of the lifting gerbe $\mathcal{G}_P$ is the obstruction class $\xi_P$ \cite{murray}. Thus, $\hat G$-lifts of $P$ exist if and only if $\mathcal{G}_P$ is trivializable.

This statement can be slightly improved by taking morphisms between bundle gerbes into account. We recall that bundle gerbes over $M$ form a 2-category $\grb AM$ \cite{stevenson1}. The Hom-category between two bundle gerbes $\mathcal{G}$ and $\mathcal{H}$ is denoted $\hom(\mathcal{G},\mathcal{H})$. A \emph{trivialization} of $\mathcal{G}$ is a 1-morphism $\mathcal{T}: \mathcal{G} \to \mathcal{I}$, where $\mathcal{I}$ denotes the trivial bundle gerbe \cite{waldorf1}.
In detail, a trivialization of our lifting gerbe $\mathcal{G}_P$ is a principal $A$-bundle $T$ over $P$ together with a bundle isomorphism
\begin{equation}
\label{triviso}
\kappa: Q \otimes \pi_2^{*}T \to \pi_1^{*}T
\end{equation}
over $P^{[2]}$ satisfying a  compatibility condition with the bundle gerbe product $\mu$.

It is a nice exercise to check that $\hat P := T$ with the projection $T \to Y \to  M$ and  the $\hat G$-action $\hat p \cdot \hat g := \kappa(\hat g^{-1} \otimes \hat p)$ is a $\hat G$-lift of $P$. 
Even better, we have

\begin{theorem}
\label{lift}
Let $P$ be a principal $G$-bundle over $M$. Then, the above  construction defines an equivalence of categories,
\begin{equation*}
\hom(\mathcal{G}_P,\mathcal{I}) \cong \struc {\hat G}P\text{.}
\end{equation*}
\end{theorem}

Lifting gerbes become even more interesting when connections are taken into account. For preparation, we  look at the Lie algebra extension
\begin{equation}
\label{lace}
\alxydim{}{0 \ar[r] & \adjust{\mathfrak{a}} \ar[r] & \adjust{\widehat{\mathfrak{g}}} \ar[r]^{t_{*}} & \adjust{\mathfrak{g}} \ar[r] & 0}
\end{equation}
 associated to the central extension \erf{ce}.
Gomi  constructed a connection on the lifting bundle gerbe $\mathcal{G}_P$ from a given connection $\eta$ on $P$ \cite{gomi3}. His construction depends on two further parameters:
\begin{enumerate}[(a)]
\item 
The first parameter is a \emph{split} $\sigma$ of the Lie algebra extension \erf{lace}, i.e. a linear map $\sigma: \mathfrak{g} \to \widehat{\mathfrak{g}}$ such that $t_{*} \circ \sigma = \id_{\mathfrak{g}}$.
As usual, one can measure the failure of $\sigma$ to be a Lie algebra homomorphism by a 2-cocycle $\omega: \mathfrak{g} \times \mathfrak{g} \to \mathfrak{a}$. Alternatively, one can lift the adjoint action of $G$ on $\mathfrak{g}$  to $\widehat{\mathfrak{g}}$, and then measure the failure of $\sigma$ to intertwine the two, resulting in the map
\begin{equation*}
Z: G \times \mathfrak{g} \to \mathfrak{a}
\quith Z(g,X) := \mathrm{Ad}_g^{-1}(\sigma(X))-\sigma(\mathrm{Ad}_g^{-1}(X))\text{.} 
\end{equation*} 

\item
The second parameter is a \emph{reduction} of $P$ with respect to the  split $\sigma$: a smooth map $r: P \times \mathfrak{g} \to \mathfrak{a}$ that is linear in the second argument and satisfies
\begin{equation}
\label{idr}
r(p,X) = r(p\cdot g,\mathrm{Ad}_{g}^{-1}(X)) - Z(g,X)
\end{equation}
for all $p\in P$, $X\in \mathfrak{g}$ and $g\in G$.

\end{enumerate}
Gomi shows that  choices of a split and a reduction with respect to it always exist, and we fix such choices for the rest of this note. 

The split $\sigma$ defines a connection $\nu$ on the principal $A$-bundle $\hat G$ over $G$, given by the formula $\nu := \theta - \sigma(t^{*}\theta) \in \Omega^1(\hat G,\hat{\mathfrak{g}})$, where $\theta$ stands for the left-invariant Maurer-Cartan form (on $\hat G$ and $G$, respectively). On the principal $A$-bundle $Q = g^{*}\hat G$ over $P^{[2]}$, we shift the pullback of the connection $\nu$ by a 1-form on $P^{[2]}$:
\begin{equation*}
\lambda_{\eta} := g^{*}\nu + Z(g,\pi_1^{*}\eta)\text{.}
\end{equation*} 
This connection on $Q$ makes the bundle gerbe product \erf{bgprod} a connection-preserving bundle morphism  \cite[Theorem 5.6]{gomi3}, and thus qualifies as the first part of a bundle gerbe connection.
It remains to define the \emph{curving}. Using the reduction $r$, we set
\begin{equation}
\label{curvingform}
C_{\eta} := -\frac{1}{2}\omega(\eta \wedge \eta) +r(\mathrm{curv}(\eta)) \in \Omega^2(P,\mathfrak{a})\text{.}
\end{equation}
The required identity for curvings,
\begin{equation}
\label{curving}
\mathrm{curv}(\lambda_{\eta}) = \pi_2^{*}C_{\eta} - \pi_1^{*}C_{\eta}\text{,}
\end{equation}
is satisfied \cite[Theorem 5.9]{gomi3}. This completes the definition of a connection on the lifting bundle gerbe $\mathcal{G}_P$. Next we  explain what this connection is good for.

First we recall what \quot{compatible connections} on $\hat G$-lifts of $P$ are. 
If $f: \hat P\to P$ is a $\hat G$-lift of $P$, a connection $\hat \eta\in \Omega^1(\hat P,\widehat{\mathfrak{g}})$ on $\hat P$ is called \emph{compatible} with the given connection $\eta$ if $f^{*}\eta = t_{*}(\hat\eta)$.
Pairs of a $\hat G$-lift $\hat P$ and a compatible connection $\hat \eta$ are called \emph{geometric $\hat G$-lifts of $P$}.
With our fixed choices of the split $\sigma$ and the reduction $r$ one can assign to any compatible connection $\hat\eta$ a scalar curvature
\begin{comment}
\footnote{One could put a minus in formula \erf{scurv} in order to avoid the minus in Theorem \ref{th:lifting}. There are reasons not to do that: firstly, formula \erf{scurv} gives the correct sign in case of complex spin structures (see Remark \erf{scurv}), and secondly, in Section \ref{sec:transgression} we will encounter another minus, so that all together there is no minus in Theorem \ref{main}.}
\end{comment}
\begin{equation}
\label{scurv}
\mathrm{scurv}(\hat\eta) := r_{\sigma}(\mathrm{curv}(\hat\eta)) \in \Omega^2(P,\mathfrak{a})\text{,}
\end{equation}
where $r_{\sigma}: \hat P \times \hat {\mathfrak{g}} \to \mathfrak{a}$ is defined by $r_{\sigma}(p,X) := X_{\mathfrak{a}} - r(f(p),X_\mathfrak{g})$ using the decomposition $X=X_{\mathfrak{a}} + \sigma(X_{\mathfrak{g}})$ of $\hat {\mathfrak{g}}$ determined by $\sigma$. The scalar curvature \erf{scurv} descends to an $\mathfrak{a}$-valued 2-form on $M$. We are interested in geometric $\hat G$-lifts with  fixed scalar curvature $\rho$; those form a category  $\struccon {\hat G}\rho P$.

In the spirit of Theorem \ref{lift} we want to compare the category $\struccon {\hat G}\rho P$ with a category of trivializations of the lifting gerbe $\mathcal{G}_P$. We recall that bundle gerbes with connection form again a 2-category $\grbcon AM$ \cite{stevenson1,waldorf1}. If $\mathcal{G}$ and $\mathcal{H}$ are bundle gerbes with connections, the \quot{connection-preserving} 1-morphisms\footnote{We put that into quotes since being connection-preserving is \emph{structure}, not a \emph{property}.} are the objects of the Hom-category $\homcon(\mathcal{G},\mathcal{H})$.
Connections on the trivial bundle gerbe $\mathcal{I}$ are given by 2-forms $\rho\in \Omega^2(M,\mathfrak{a})$. We denote the trivial bundle gerbe with connection $\rho$ by $\mathcal{I}_{\rho}$. \emph{Flat} trivializations of $\mathcal{G}$ are the objects of the category $\homcon(\mathcal{G},\mathcal{I}_0)$. 

Gomi proves \cite[Theorem 3.9 and Corollary 5.13]{gomi3} that $\mathcal{G}_P$ has a \emph{flat} trivialization if and only if there exist a geometric $\hat G$-lift with  \emph{vanishing} scalar curvature.  We need the following generalization to arbitrary scalar curvature, whose proof we leave as an exercise in Lie-algebra valued differential forms. 

\begin{theorem}
\label{th:lifting}
Let $P$ be a principal $G$-bundle over $M$ with connection, and  $\rho\in \Omega^2(M,\mathfrak{a})$. Then, the equivalence of Theorem \ref{lift} extends to an equivalence of categories
\begin{equation*}
\homcon(\mathcal{G}_P,\mathcal{I}_{\rho}) \cong \struccon{\hat G}{-\rho}{P}\text{.} \end{equation*}
\end{theorem}

\begin{comment}
If a trivialization $\mathcal{T}: \mathcal{G}_P \to \mathcal{I}_{\rho}$ is given, with a principal $A$-bundle $\chi:T \to P$ with connection $\omega$ and a bundle isomorphism $\kappa$, a connection on $\hat P := T$ is defined by
\begin{equation*}
\hat\eta := \omega + \sigma(\chi^{*}\eta) \in \Omega^1(\hat P,\hat{\mathfrak{g}})\text{.}
\end{equation*}
It is clear the compatibility condition $t_{*}(\hat\eta) = f^{*}\eta$ is satisfied, and the proof that $\hat\eta$ is a connection uses the definition of $\lambda_{\eta}$ and the relation between $Z$ and $\sigma$. Its scalar curvature is $-\rho$. 

Conversely, if $\hat\eta$ is a compatible connection on $\hat P$ with scalar curvature $-\rho$, then $\omega := \hat\eta - \sigma(\chi^{*}\eta) \in \Omega^1(T,\mathfrak{a})$ is the desired connection on $T$. 
\end{comment}

\section{Transgression and Fusion Bundles}

\label{sec:transgression}

In this section we discuss a relation between bundle gerbes over $M$ and bundles with fusion products over $LM$. Employing this relation for the lifting bundle gerbe $\mathcal{G}_P$ from the previous section yields the bundle $\mathscr{L}_P$ that appears in Theorem \ref{main}.

We denote by $\grbcon AM$ the 2-category of bundle gerbes with connection over $M$, and by $\hc 1 \grbcon AM$ its \quot{homotopy 1-category} whose morphisms are 2-isomorphism classes of the former 1-morphisms. Further, we denote by $\buncon A{LM}$ the category of principal $A$-bundles with connection over $LM$. The main idea we need in this section is a functor 
\begin{equation*}
L: \hc 1 \grbcon AM \to \buncon A {LM}\text{.}
\end{equation*}
This functor realizes -- on the level of characteristic classes -- the transgression  homomorphism
\begin{equation}
\label{transhom}
\h^2(M,\underline{A}) \to \h^1(LM,\underline{A})
\end{equation}
in the cohomology with values in the sheaf of smooth $A$-valued functions. 

A first version of the functor $L$ has been described by Brylinski and McLaughlin in the language of Dixmier-Douady sheaves of groupoids and line bundles \cite{brylinski1,brylinski4}. We are going to review it briefly in the language of bundle gerbes; for a detailed treatment I refer to \cite[Section 3.1]{waldorf5} and \cite[Section 4]{waldorf10}.  

Given a bundle gerbe $\mathcal{G}$ with connection over $M$, the principal $A$-bundle $L\mathcal{G}$ over $LM$ is defined as follows. Over a loop $\tau\in LM$, its fibre consists of isomorphism classes of flat trivializations of $\tau^{*}\mathcal{G}$, i.e.
\begin{equation*}
L\mathcal{G}_{\tau} := \hc 0 \homcon (\tau^{*}\mathcal{G},\mathcal{I}_0)\text{.}
\end{equation*}
The $A$-action on these fibres is induced by tensoring the principal $A$-bundle $T$ of a trivialization $\mathcal{T}: \tau^{*}\mathcal{G} \to \mathcal{I}_0$ with the pullback of a principal $A$-bundle $P_a$ over $S^1$ with $\mathrm{Hol}_{P_a}(S^1)=a$.
There exists a Fréchet manifold structure on $L\mathcal{G}$ that makes this a smooth principal $A$-bundle over $LM$ \cite[Proposition 3.1.2]{waldorf5}.

The connection on $L\mathcal{G}$ can be defined by prescribing its parallel transport, see \cite{schreiber3,waldorf9}. Suppose $\gamma$ is a path in $LM$, and $\mathcal{T}_0 \in L\mathcal{G}_{\tau_0}$ and $\mathcal{T}_1 \in L\mathcal{G}_{\tau_1}$ are trivializations of $\mathcal{G}$ over the end-loops of $\gamma$. We look at the associated cylinder $\exd\gamma: [0,1] \times S^1 \to M$. As a bundle gerbe with connection, $\mathcal{G}$ associates to this cylinder a surface holonomy  $\mathrm{Hol}_{\mathcal{G}}(\exd\gamma,\mathcal{T}_0,\mathcal{T}_1)\in A$, where the two trivializations act as boundary conditions (\quot{D-branes}), see \cite{carey2}. We put
\begin{equation}
\label{defpartrans}
\tau_{\gamma}: L\mathcal{G}_{\tau_0} \to L\mathcal{G}_{\tau_1}:  \mathcal{T}_0 \mapsto \mathcal{T}_1 \cdot \mathrm{Hol}_{\mathcal{G}}(\exd\gamma,\mathcal{T}_0,\mathcal{T}_1)\text{.} \end{equation}
One can show that this prescription indeed defines indeed a connection on $L\mathcal{G}$ \cite[Proposition 4.3.2]{waldorf10}.

Concerning the morphisms, consider a  1-isomorphism $\mathcal{A}: \mathcal{G} \to \mathcal{H}$  between bundle gerbes with connections. Over a loop $\tau\in LM$, it induces the  morphism
\begin{equation*}
L\mathcal{A}: L\mathcal{G} \to L\mathcal{H}: \mathcal{T} \mapsto \mathcal{T} \circ \tau^{*}\mathcal{A}^{-1}\text{,}
\end{equation*}
where $\mathcal{T}: \tau^{*}\mathcal{G} \to \mathcal{I}_0$ is a trivialization of $\tau^{*}\mathcal{G}$, and $\circ$ and $()^{-1}$ denote the composition and the inversion, respectively, of 1-isomorphisms in the 2-category $\grbcon A{S^1}$. This yields a smooth, connection-preserving bundle morphism, and so finishes the definition of the transgression functor $L$.

Brylinski and McLaughlin use the functor $L$ to study geometrically the image of the cohomological transgression homomorphism \erf{transhom} \cite{brylinski1,brylinski4}. They argue that all bundles in the image of the functor $L$ are automatically equipped with one of the \quot{products} from the beginning of the introduction to this note. We describe a slightly modified version of this product called \emph{fusion product}. 

We look at the set $PM$ of smooth maps $\gamma: [0,1] \to M$ that are locally constant at $\left \lbrace 0,1 \right \rbrace$ (they have \quot{sitting instants}) and equip that set with the evaluation map $\ev:PM \to M\times M$. In \cite{waldorf10} I treat spaces of paths with sitting instants rigorously in the framework of generalized manifolds. In this note we will pretend they were Fréchet manifolds; this will lead us in the end to correct statements. The evaluation map is a \quot{surjective submersion} and we have the fibre products $PM^{[k]}$ available; employing the notation for fibre products introduced in Section \ref{sec:lifting}. Explicitly, a point in $PM^{[k]}$ is a $k$-tuple of paths in $M$ with common endpoints. 

The sitting instants permit to define a smooth map 
\begin{equation*}
l: PM^{[2]} \to LM:(\gamma_1,\gamma_2) \mapsto \prev{\gamma_2} \pcomp \gamma_1\text{,}
\end{equation*}
where $\prev{\gamma}$ denotes the reversed path and $\pcomp$ denotes the concatenation of paths. Combining this map with the projections $\ev_{ij}: PM^{[3]} \to PM^{[2]}$ we obtain the smooth maps $e_{ij} := l \circ  \ev_{ij}$. Now we are in the position to give the central definition of this note.

\begin{definition}[{{\cite[Definition 2.1.3]{waldorf10}}}]
\label{def:fusionproduct}
A \emph{fusion product} on a principal $A$-bundle $P$ over $LM$ is a smooth bundle isomorphism
\begin{equation*}
\lambda : e_{12}^{*}P \otimes e_{23}^{*}P \to e_{13}^{*}P
\end{equation*}
over $PM^{[3]}$ that is associative over $PM^{[4]}$.  
\end{definition}

It is not totally trivial to spot the fusion product on a transgressed principal $A$-bundle $L\mathcal{G}$. It can be characterized as follows --- for a detailed treatment see \cite[Section 4.2]{waldorf10}. For $(\gamma_1,\gamma_2,\gamma_3)$ an element in the space $PM^{[3]}$, we write $\tau_{ij} := l(\gamma_1,\gamma_2) \in LM$. Pick trivializations $\mathcal{T}_{ij} \in L\mathcal{G}_{\tau_{ij}}$ over these loops. We introduce the maps $\iota_1,\iota_2: [0,1] \to \R/\Z$ defined by $\iota_1(t) := \frac{1}{2}t$ and $\iota_2(t)=1-\frac{1}{2}t$. Pullback along $\iota_1$ and $\iota_2$ \quot{restricts} the trivializations $\mathcal{T}_{ij}$ to intervals, all together giving two trivializations over each of the paths $\gamma_k$. Now we pick 2-isomorphisms
\begin{equation*}
\phi_1 : \iota_1^{*}\mathcal{T}_{12} \Rightarrow \iota_1^{*}\mathcal{T}_{13}
\quomma
\phi_2: \iota_2^{*}\mathcal{T}_{12} \Rightarrow \iota_1^{*}\mathcal{T}_{23}
\quand
\phi_3: \iota_2^{*}\mathcal{T}_{23} \Rightarrow \iota_2^{*}\mathcal{T}_{13}\text{;}
\end{equation*}
these always exist and the notation is such that the 2-isomorphism $\phi_k$ is over the path $\gamma_k$. Let $x$ be the common initial point and $y$ be the common end point of the paths $\gamma_k$. All three 2-isomorphisms can be restricted to $x$ and to $y$, and we have
\begin{equation}
\label{lambda}
\lambda(\mathcal{T}_{12} \otimes \mathcal{T}_{23}) = \mathcal{T}_{13}
\end{equation}
if and only if the relation $\phi_1 = \phi_3 \circ \phi_2$ holds over both $x$ and $y$.

The work of Brylinski and McLaughlin \cite{brylinski1,brylinski4} suggests that the existence of fusion products characterizes bundles in the image of the transgression functor $L$  among all principal $A$-bundles over $LM$. The main result of my paper \cite{waldorf10} (Theorem \ref{th:transgression} below) shows that this is  true if one requires additionally a connection on $P$ satisfying three conditions: it has to be compatible, symmetrizing and superficial. The easiest of these is the \emph{compatibility} with the fusion product: it simply means that the fusion product $\lambda$ is a connection-preserving bundle morphism. The second condition requires that the connection \emph{symmetrizes} the fusion product in a subtle way, and the third condition imposes  constraints on its holonomy. Since these conditions will not appear explicitly in the following, I omit a more detailed discussion.

\begin{definition}[{{\cite[Definition A]{waldorf10}}}]
\label{fusbun}
A \emph{fusion bundle with superficial connection} over $LM$ is a principal $A$-bundle $P$ over $LM$ with a fusion product and a compatible, symmetrizing and superficial connection. 
\end{definition}

The principal $A$-bundle $L\mathcal{G}$ equipped with the connection and the fusion product constructed above is such a fusion bundle with connection. To see this, one has to check the three conditions ---  this is  quite tedious and can be found in \cite[Sections 4.2 and 4.3]{waldorf10}. 

Let us look at the category $\fusbuncon A {LM}$ composed of fusion bundles with connection and connection-preserving, fusion-preserving bundle isomorphisms. As we have motivated above, the transgression functor $L$ \emph{lifts} to this category  as an \quot{improved} transgression functor
\begin{equation*}
\tr: \hc 1 \grbcon A M \to \fusbuncon A {LM}\text{.}
\end{equation*}
In the sense of the following theorem, this functor captures \emph{all} features of transgression.

\begin{theorem}[{{\cite[Theorem A]{waldorf10}}}]
\label{th:transgression}
Let $M$ be a connected smooth manifold. Then, the  improved transgression  functor $\mathscr{T}$
is an  equivalence of  categories. 
\end{theorem}

To close this section about transgression let us compute the transgression $\mathscr{T}_{\mathcal{I}_{\rho}}$ of the trivial bundle gerbe $\mathcal{I}$ equipped with the connection defined by a 2-form $\rho\in\Omega^2(M,\mathfrak{a})$. \begin{lemma}
\label{lem:trivtrans}
The fusion bundle with connection $\mathscr{T}_{\mathcal{I}_{\rho}}$ has a canonical, fusion-preserving section $\sigma: LM \to \mathscr{T}_{\mathcal{I}_{\rho}}$ of curvature $L\rho$, i.e.
\begin{equation*}
\lambda(\sigma(\prev{\gamma_2} \pcomp \gamma_1) \otimes \sigma(\prev{\gamma_3} \pcomp \gamma_2)) = \sigma(\prev{\gamma_3} \pcomp \gamma_1)
\quand 
\sigma^{*}\omega =  \int_{S^1} \ev^{*}\rho
\end{equation*}
for all $(\gamma_1,\gamma_2,\gamma_3) \in PM^{[3]}$, where $\lambda$ is the fusion product, $\omega$ the connection on $\mathscr{T}_{\mathcal{I}_{\rho}}$ and $\ev: S^1 \times LM \to M$ is the evaluation map. 
\end{lemma}

\begin{comment}
In the published version of this article it is claimed that the curvature of the section
$\sigma$ is $-L\rho$; this claim missed the sign in the last equality of \erf{eq:complicated}, which comes from the fact  passing between these two integrals permutes the two tangent vectors the 2-form $\rho$ depends on.
\end{comment}

\begin{proof}
The section itself is defined  by the fact that we have over every loop $\tau\in LM$ a distinguished element in the fibre $L\mathcal{I}_{\rho}$, namely $\tau^{*}\id_{\mathcal{I}_{\rho}}$. That this section is fusion-preserving is straightforward to see using the  characterization \erf{lambda} of $\lambda$ and by choosing $\mathcal{T}_{ij} := \tau_{ij}^{*}\id_{\mathcal{I}_{\rho}}$. For the curvature we calculate for a path $\gamma : [0,1] \to LM$:
\begin{equation}
\label{eq:complicated}
\exp\left ( - \int_{\gamma} \sigma^{*}\omega \right ) = \mathrm{Hol}_{\mathcal{I}_{\rho}}(\exd\gamma,\sigma(\gamma(0)),\sigma(\gamma(1))) = \exp \left ( \int_{\exd\gamma}\rho \right ) = \exp \left ( -\int_{\gamma}\int_{S^1} \ev^{*}\rho \right )\text{.}
\end{equation}
The first equality is the relation between a connection 1-form and its parallel transport \erf{defpartrans} --- the minus  is a convention of \cite{schreiber3} that is required to consistently translate between connection 1-forms and their parallel transport. The second equality calculates the surface holonomy of a trivial bundle gerbe, and the third is obtained by splitting  the integration over $[0,1] \times S^1$ into two integrals.
\begin{comment}
This is the calculation:
\begin{eqnarray*}
-\int_{\gamma^{\vee}}\rho &=&- \int_{S^1 \times S^1} (\gamma^{\vee})^{*}\rho 
\\&=&- \int_{0^1}\mathrm{d}z_1 \int_{0}^1\mathrm{d}z_2 \;(\gamma^{\vee})^{*}\rho(\partial_{z_1},\partial_{z_2})
\\&=& - \int_{0}^1\mathrm{d}z_1 \int_{0}^1\mathrm{d}z_2 \;\rho_{\gamma^{\vee}(z_1,z_2)}(\partial_{z_1}\gamma^{\vee}(z_1,z_2),\partial_{z_2}\gamma^{\vee}(z_1,z_2))
\\&=& \int_0^1 \mathrm{d}z_1\; \int_0^1 \mathrm{d}z_2\; \rho_{\gamma(z_1)(z_2)}(\partial_{z_2}\gamma(z_1)(z_2),\partial_{z_1}\gamma(z_1)(z_2))
\\&=& \int_0^1 \mathrm{d}z_1\; \tau_{\Omega}(\rho)_{\gamma(z_1)}(\partial_{z_1}\gamma(z_1)) 
\\&=& \int_{\gamma} \tau_{\Omega}(\rho)\text{.}
\end{eqnarray*}
\end{comment}
But 1-forms are characterized uniquely by their integrals along paths \cite[Theorem B.2]{waldorf9} --- this finishes the proof. 
\end{proof}

\setsecnumdepth{1}

\section{The Bundle $\mathscr{L}_P$ over the Loop Space}

\label{sec:proof}

In this section we combine the main results of Sections \ref{sec:lifting} and \ref{sec:transgression} and prove Theorem \ref{main}. Let $P$ be a principal $G$-bundle with connection over $M$, and let $\sigma$ and $r$ be our choices of a split and a reduction, respectively, so that a lifting bundle gerbe $\mathcal{G}_P$ with connection is determined. We set
\begin{equation}
\label{lp}
\mathscr{L}_P := \mathscr{T}_{\mathcal{G}_P}\text{.}
\end{equation}

With this definition, we are  in a position  to give the proof of Theorem \ref{main}. We use that the equivalences from Theorems \ref{th:lifting} and \ref{th:transgression} induce bijections on isomorphism classes and Hom-sets, respectively, and obtain bijections
\begin{equation*}
\hc 0 (\struccon {\hat G}\rho P) \cong \hc 0 \homcon (\mathcal{G}_P,\mathcal{I}_{-\rho}) \cong \mathrm{Hom}(\mathscr{L}_{P},\mathscr{T}_{\mathcal{I}_{-\rho}})\text{.} 
\end{equation*}
Using the canonical section of $\mathscr{T}_{\mathcal{I}_{-\rho}}$ from Lemma \ref{lem:trivtrans} we can identify the set on the right hand side with the set of fusion-preserving sections of $\mathscr{L}_P$ of curvature $-L\rho$. 
This is Theorem \ref{main}.

Let us look in more detail what the bundle $\mathscr{L}_P$ defined in \erf{lp} is, under the identification of geometric $\hat G$-lifts of $P$ and trivializations of $\mathcal{G}_P$ of Theorem \ref{th:lifting}. Over a loop $\tau \in LM$, the fibre of $\mathscr{L}_P$ consists of all geometric  $\hat G$-lifts of $\tau^{*}P$. The fusion product on $\mathscr{L}_P$ can be seen as a structure that \emph{glues}  geometric $\hat G$-lifts over loops $\prev{\gamma_2} \pcomp \gamma_1$ and $\prev{\gamma_3} \pcomp \gamma_2$ to a third geometric $\hat G$-lift over the loop $\prev{\gamma_3} \pcomp \gamma_1$. The connection can  be described as follows: if $\gamma$ is a path in $LM$ with end-loops $\tau_0$ and $\tau_1$, then geometric  $\hat G$-lifts $\hat P_0$ of $\tau_0^{*}P$ and $\hat P_1$ of $\tau_1^{*}P$ are related by parallel transport along $\gamma$, if and only if $\hat P_0$ and $\hat P_1$ are restrictions of a geometric $\hat G$-lift $(\hat P,\hat\eta)$ over the cylinder $\exd\gamma:S^1 \times [0,1] \to M$ with
\begin{equation*}
\exp \left ( \int_{\exd\gamma} \mathrm{scurv}(\hat\eta) \right ) = 1\text{.}
\end{equation*}

Finally, with the above explanations, the map in Theorem \ref{main}  assigns to a geometric $\hat G$-lift $\hat P$ of $P$ the section $\sigma_{\hat P}: LM \to \mathscr{L}_P$ given by $\sigma_{\hat P}(\tau) := \tau^{*}\hat P$. Theorem \ref{main} states the exact conditions under which one can go in the opposite direction, i.e. when \quot{loop-wise} geometric $\hat G$-lifts of $P$ patch together to a \quot{global} one.

\section{An Alternative Construction of $\mathscr{L}_P$}

\label{sec:alternative}

In this section we present an alternative construction of the bundle $\mathscr{L}_P$ that does not use any gerbe theory.
We start  with a principal $G$-bundle $P$ over $M$ with connection $\eta$, and fix choices of a split $\sigma$ and a reduction $r$ with respect to $\sigma$. 

As a prerequisite we recall the following fact. 

\begin{lemma}
\label{lem:looping}
Let $G$ be a connected Lie group, and let $\pi:P\to M$ be a principal $G$-bundle over $M$. Then, $L\pi: LP \to LM$ is a Fréchet principal $LG$-bundle over $LM$.
\end{lemma}

\begin{proof}
The assumption that $G$ is connected assures that $L\pi: LP \to LM$ is surjective. One chooses a connection on $P$ and lifts a loop $\tau$ in $M$ horizontally to a path $\gamma$ in $P$. Then one acts on it with a path $\beta$ in $G$ connecting $1$  with the difference between $\gamma(0)$ and $\gamma(1)$. The new path $\gamma\cdot\beta$ is closed, and by choosing $\beta$ with sitting instants, horizontal in a neighborhood of $\left \lbrace 0,1  \right \rbrace$, in particular, it is a  smooth loop. The proof is completed  in \cite[Proposition 1.9]{spera1}.
\end{proof}

\begin{comment}
In more detail, we choose a connection on $P$. For a point $\tau_0\in LM$, we  choose a  local section $s:U \to P$ defined in an open neighborhood $U$ of the base point $\tau_0(1)$. Let $V \subset LM$ denote the open set of loops based in $U$. For $\tau\in V$, we denote by $\tilde\tau: [0,1] \to P$ its horizontal lift with $\tilde\tau(0)=s(\tau(1))$. Taking the \quot{difference} between its endpoints, we obtain a smooth map $h: V \to G$. 
Since $G$ is connected by assumption, there exists an open neighborhood $W$ of $\tau_0$ such that $h$ lifts to $H: W \to P_eG$ along the endpoint evaluation. We claim that for all $\tau\in W$ the path
\begin{equation*}
\gamma_\tau:[0,1] \to P: t \mapsto \tilde\tau(t)\cdot H(\tau)(t)
\end{equation*}
is in fact a  loop; this finishes the construction of a local section $\gamma: W \to LP$. Since we have $\gamma_{\tau}(0)=\gamma_{\tau}(1)$ by construction, to prove the claim it remains to check that $\gamma_{\tau}$ glues smoothly. As in Section \ref{sec:transgression} we have assumed that the paths in the image of $H$ have sitting instants. Thus, $H$ is constant in the neighborhood of the gluing point. But then, $\gamma_{\tau}$ is horizontal in this neighborhood, and thus smooth.   
\end{comment}

By Lemma \ref{lem:looping} the given bundle $P$ defines  a surjective submersion $L\pi: LP\to LM$, and we have  fibre products $LP^{[k]}$ available. We remark that taking loops commutes with taking fibre products in the sense of canonical diffeomorphisms $L(P^{[k]}) \cong LP^{[k]}$, and we will not further distinguish between these two manifolds. From Section \ref{sec:lifting} we take the  principal $A$-bundle $Q$ over $P^{[2]}$ with its connection $\lambda_{\eta}$, and denote by $g: LP^{[2]} \to A$ its holonomy. The multiplicativity of  $Q$ from \erf{bgprod} implies the cocycle condition 
\begin{equation*}
L\pi_{12}^{*}g \cdot L\pi_{23}^{*}g = L\pi_{13}^{*}g
\end{equation*}
over $LP^{[3]}$. 
We regard $g$ as a \v Cech cocycle with respect to the \quot{cover} $L\pi: LP \to LM$, and apply the usual reconstruction of principal bundles from cocycles:
\begin{equation*}
\mathscr{L}_P^{\,\prime} := (LP \times A) \;/\; \sim_g
\quith
(\tau',a) \sim_g (\tau,g(\tau,\tau')\cdot a)
\end{equation*}
is a principal $A$-bundle over $LM$. This is our alternative construction of $\mathscr{L}_P$. Before we show that $\mathscr{L}_P^{\,\prime}$ and $\mathscr{L}_P$ are isomorphic, we continue with specifying connection and fusion product on $\mathscr{L}_P^{\,\prime}$.

In our \v Cech picture, a connection on $\mathscr{L}_P^{\,\prime}$ is determined by a \quot{local} 1-form $\omega \in \Omega^1(LP,\mathfrak{a})$ that is compatible with the cocycle $g$ in the sense that $L\pi_2^{*}\omega = L\pi_1^{*}\omega + g^{*}\theta$ over $LP^{[2]}$. 
\begin{comment}
Indeed, the actual connection 1-form on $\mathscr{L}_P'$ is
$\tilde\omega_{(\tau,a)} := \omega_{\tau} + \theta_{a}$.
\end{comment}
We choose
$\omega := LC_{\eta}$, the transgression of the curving 2-form $C_{\eta}$ from \erf{curvingform}. Indeed, we find
\begin{equation*}
L\pi_2^{*}\omega = L\pi_1^{*}\omega + L\mathrm{curv}(\lambda_{\eta}) =L\pi_1^{*}\omega + \mathrm{dlog}(\mathrm{Hol}_Q) = L\pi_1^{*}\omega + g^{*}\theta\text{.}
\end{equation*}
The first equality is the transgression of Eq. \erf{curving}, and the second equality uses the fact that the derivative of the holonomy of a connection $\lambda_{\eta}$ is equal to  the transgression of its curvature, see e.g. \cite[Prop. 3.2.13]{waldorf9}.

Finally, we equip the bundle $\mathscr{L}_P^{\,\prime}$ with a fusion product. The \v Cech description of a fusion product is not particularly nice but  possible. We  consider a triple $(\gamma_1,\gamma_2,\gamma_3) \in PM^{[3]}$ of paths with common endpoints, and lifts $\tilde \beta_{ij}\in LP$ of the three associated loops $\beta_{ij} := \ell(\gamma_i,\gamma_j) \in LM$. Such data form the space $Z := PM^{[3]} \times_{L M^{3}} LP^3$. A fusion product is given by a smooth map
$f: Z \to A$ such that --- if $\tilde\beta_{ij}'$  are different lifts of the loops $\beta_{ij}$ --- one has
\begin{equation}
\label{fusid1}
g(\tilde\beta_{12},\tilde\beta_{12}') \cdot g(\tilde\beta_{23},\tilde\beta_{23}') \cdot f(\tilde\beta_{12}',\tilde\beta_{23}',\tilde\beta_{13}') \\= f(\tilde\beta_{12},\tilde\beta_{23},\tilde\beta_{13}) \cdot g(\tilde\beta_{13},\tilde\beta_{13}')\text{.}
\end{equation}
Additionally, the associativity condition for the fusion product requires that for four paths $(\gamma_1,\gamma_2,\gamma_3,\gamma_4)\in PM^{[4]}$ and accordant lifts $\tilde \beta_{ij}$ we have
\begin{equation}
\label{fusid2}
f(\tilde \beta_{13},\tilde \beta_{34},\tilde \beta_{14}) \cdot f(\tilde \beta_{12},\tilde \beta_{23},\tilde \beta_{13}) = f(\tilde \beta_{12},\tilde \beta_{24},\tilde \beta_{14})\cdot f(\tilde \beta_{23},\tilde \beta_{34},\tilde \beta_{24})\text{.}
\end{equation}

In the present situation of the bundle $\mathscr{L}_P^{\,\prime}$, the map $f:Z \to A$ representing the fusion product is produced in the following way. We split each loop $\beta_{ij}$ into two paths
\begin{equation*}
\mu_{ij}(t) := \beta_{ij}(\textstyle\frac{1}{2}t)
\quand
\nu_{ij}(t) := \beta_{ij}(1-\textstyle\frac{1}{2}t)\text{.}
\end{equation*}
We combine these to three paths
\begin{equation*}
\tilde \gamma_1 := (\mu_{12},\mu_{13})
\quomma
\tilde\gamma_2 := (\nu_{12},\mu_{23})
\quand
\tilde\gamma_3 := (\nu_{23},\nu_{13})
\end{equation*}
in $P^{[2]}$, and the indices are chosen such that each path  $\tilde\gamma_k$ sits over the original path $\gamma_k$ in $M$. We want to measure the failure of the parallel transport $\tau_{\tilde\gamma_k}$ in the bundle $Q$ over $P^{[2]}$  to respect the isomorphism $\mu$ from Eq. \erf{bgprod}. For that purpose, we choose elements $\hat g_{2} \in Q_{\tilde\gamma_2(0)}$ and $\hat g_{3} \in Q_{\tilde\gamma_3(0)}$ and set $\hat g_{1} := \mu(\hat g_{2} \otimes \hat g_{3}) \in Q_{\tilde\gamma_1(0)}$. Now we define $f$ such that
\begin{equation*}
\mu(\ptr{\tilde\gamma_2}(\hat g_{2}) \otimes \ptr{\tilde\gamma_3}(\hat g_{3})) \cdot  f(\beta_{12},\beta_{23},\beta_{13}) = \ptr{\tilde\gamma_1}(\hat g_{1}) \text{.}
\end{equation*}
Due to the $A$-equivariance of parallel transport,  this definition is independent of the choices of $\hat g_{2}$ and $\hat g_{3}$. Since $\hat g_{2}$ and $\hat g_3$ can be chosen locally in a smooth way, $f$ is a smooth map. Checking the identities \erf{fusid1} and \erf{fusid2} is a straightforward calculation that we leave out for brevity.

Summarizing, we have defined a principal $A$-bundle $\mathscr{L}_P^{\,\prime}$ over $LM$ with a connection and a fusion product. As a consequence of Theorem \ref{prop:transgression} below,  $\mathscr{L}_P^{\,\prime}$ is in fact a fusion bundle with connection in the sense of Definition \ref{fusbun}.

\begin{theorem}
\label{prop:transgression}
There exists a canonical, connection-preserving and fusion-preserving bundle isomorphism $\eta: \mathscr{L}_P \to \mathscr{L}_P^{\,\prime}$.
\end{theorem}

\begin{proof}
We recall that a point in  $\mathscr{L}_P = \mathscr{T}_{\mathcal{G}_P}$ over a loop $\tau \in LM$ is a flat trivialization of $\tau^{*}\mathcal{G}_P$. In turn, this is a principal $A$-bundle $T$ over $\tau^{*}P$ with connection, whose holonomy we denote by $h: L(\tau^{*}P) \to A$. Further, there is an isomorphism (see \erf{triviso}) of principal bundles over $(\tau^{*}P)^{[2]}$ which guarantees
\begin{equation}
\label{fuck}
g \cdot L\pi_2^{*}h = L\pi_1^{*}h\text{,}
\end{equation}
where $g$ is the \v Cech cocycle used above. Now let $\tilde\tau\in LP$ be a lift of $\tau$, and let a loop $\beta_{\tilde\tau} \in L(\tau^{*}P)$ be defined by $\beta_{\tilde\tau}(z) := (z,\tilde\tau(z)) \in S^1 \lli{\tau} \times_{\pi}P$. Eq. \erf{fuck} shows that the class $(\tilde\tau,h(\beta_{\tilde\tau})) \in \mathscr{L}_P^{\,\prime}$ is independent of the choice of $\tilde\tau$. Since we can choose the lifts $\tilde\tau$ locally smooth, this defines a smooth map
\begin{equation*}
\eta: \mathscr{L}_P \to \mathscr{L}_P^{\,\prime}\text{.}
\end{equation*}
By construction, it respects the projections to $LM$ and is $A$-equivariant.

In order to see that $\eta$ is connection-preserving, we compare the parallel transports in $\mathscr{L}_P$ and $\mathscr{L}_P^{\,\prime}$ along a path $\gamma$ in $LM$. By Lemma \ref{lem:looping} we may choose a lift $\tilde\gamma$ to $LP$. Accordingly, the associated map $\exd\gamma: C \to M$ on the cylinder $C=[0,1]\times S^1$ lifts to a map $\exd{\tilde\gamma}: C \to P$. This lift defines a section into the submersion of the bundle gerbe $(\exd\gamma)^{*}\mathcal{G}_P$ over $C$, and any such section determines a trivialization $\mathcal{T}: (\exd\gamma)^{*}\mathcal{G}_P \to \mathcal{I}_{\rho}$, where $\rho = (\exd{\tilde\gamma})^{*}C_{\eta}$. \cite[Lemma 3.2.3]{waldorf10}. We use this trivialization to compute
\begin{equation*}
\mathrm{Hol}_{\mathcal{G}}(\exd\gamma,\mathcal{T}_0,\mathcal{T}_1) = \exp \left ( \int_{C} \rho \right )= \exp \left ( \int_{\exd{\tilde\gamma}} C_{\eta} \right ) =: a \in A\text{.}
\end{equation*}
Putting $\mathcal{T}_0 := \mathcal{T}|_{\gamma(0)}$ and $\mathcal{T}_1 := \mathcal{T}|_{\gamma(1)}$ we get for the parallel transport 
\begin{equation*}
\tau_{\gamma}: \mathscr{L}_P|_{\gamma(0)} \to \mathscr{L}_P|_{\gamma(1)}: \mathcal{T}_0 \mapsto \mathcal{T}_1 \cdot a\text{.}
\end{equation*}
Further, the loops $\beta_{\tilde\gamma(0)}$ and $\beta_{\tilde\gamma(1)}$ have trivial holonomy in the bundle $Q$, since they factor through the diagonal $P \to P^{[2]}$ over which $Q$ has a flat section. This shows that
$\eta(\mathcal{T}_0) := (\tilde\gamma(0),1)$
and
$\eta(\mathcal{T}_1) := (\tilde\gamma(1),1)$.
On the other side, the parallel transport in $\mathscr{L}_P^{\,\prime}$ is \begin{equation}
\label{pt2}
\tau_{\gamma}: \mathscr{L}_P^{\,\prime}|_{\gamma(0)} \to \mathscr{L}_P^{\,\prime}|_{\gamma(1)}
:
\eta(\mathcal{T}_0) \mapsto \eta(\mathcal{T}_1) \cdot \exp \left ( \int_{[0,1]} \tilde\gamma^{*}\omega  \right )\text{.}
\end{equation}
Looking at the definition of the connection $\omega$, the  integral in \erf{pt2} coincides with the constant $a$, and this shows that the isomorphism $\eta$ commutes with parallel transport.

In order to see that the isomorphism $\eta$ is fusion-preserving, let $(\gamma_1,\gamma_2,\gamma_3) \in PM^{[3]}$ with associated loops $\beta_{ij} := l(\gamma_i,\gamma_j)\in LM$ and lifts $\tilde\eta_{ij} \in LP$. We pick trivializations $\mathcal{T}_{ij}$ over $\beta_{ij}$ such that $\lambda(\mathcal{T}_{12},\mathcal{T}_{23}) = \mathcal{T}_{13}$, and denote their images under $\eta$ by $\eta(\mathcal{T}_{ij}) = (\tilde\beta_{ij},h_{ij})$. In this notation, the claim that $\eta$ is fusion-preserving boils down to the equality
\begin{equation*}
h_{12} \cdot h_{23} \cdot f(\tilde\beta_{12},\tilde\beta_{23},\tilde\beta_{13}) = h_{13}\text{.}
\end{equation*}
It can be verified by a tedious but straightforward computation with parallel transport in the bundles $T_{ij}$ that belong to the trivializations $\mathcal{T}_{ij}$, involving the 2-isomorphisms $\phi_k$ that make up the fusion product on $\mathscr{L}_P$ in the sense of Eq. \erf{lambda}. 
\end{proof}

\setsecnumdepth{1}

\section{Spin Structures and Loop Space Orientations}

\label{sec:spin}

In this section we discuss two applications of Theorem \ref{main}: spin structures and complex spin structures on an $n$-dimensional oriented Riemannian manifold $M$.

\subsection{First Example: Spin Structures}

We are concerned with the central extension
\begin{equation}
\label{spin}
\alxydim{}{1 \ar[r] & \Z_2 \ar[r] & \spin n \ar[r] & \so n \ar[r] & 1}
\end{equation}
and with the $\so n$-bundle $FM$ of orthonormal frames in an $n$-dimensional oriented Riemannian manifold $M$. A \emph{spin structure} on $M$ is precisely a $\spin n$-lift of $FM$. Since $\Z_2$ is discrete, all connections and differential forms disappear from the statement of Theorem \ref{main}, so that equivalence classes of spin structures are in bijection with fusion-preserving sections of $\mathscr{L}_{FM}$. 

Let us look closer at the bundle $\mathscr{L}_{FM}$ in its alternative formulation of from Section \ref{sec:alternative}.
Its total space is 
\begin{equation}
\label{lfm}
\mathscr{L}_{FM} = (LFM \times \Z_2)\;/\; L\so n
\end{equation}
where $L\so n$ acts on $\Z_2$ via the monodromy  $m: L\so n \to \Z_2$ of the spin extension \erf{spin}. This bundle can be seen as the \emph{orientation bundle} of the $LM$ for various reasons, of which two are:
\begin{enumerate}[(i)]
\item 
The $L\so n$-bundle $LFM$ over $LM$ plays the role of the frame bundle of $LM$. According to Eq. \erf{lfm}, our bundle $\mathscr{L}_{FM}$ is a $\Z_2$-reduction of that frame bundle, just like the orientation bundle of a finite-dimensional manifold, cf. \cite{mclaughlin1}. 

\item
Since the transgression functor $\mathscr{T}$ covers the ordinary transgression homomorphism \erf{transhom} on the level of cohomology, and the characteristic class of the lifting gerbe $\mathcal{G}_{FM}$ is the second Stiefel-Whitney class $w_2 \in \h^2(M,\Z_2)$, we see that the characteristic class of $\mathscr{L}_{FM}$ is the transgression of $w_2$ to $LM$ --- Atiyah has defined that class to be the obstruction against orientability of $LM$ \cite[the remark after Lemma 3]{atiyah2}. 
\end{enumerate}
In that respect, we  see the sections of $\mathscr{L}_{FM}$ as \emph{orientations} of the loop space $LM$.

The new information that enters this picture from the general point of view of this note (but also from a different perspective involving Clifford bimodules \cite{stolz3}) is that the orientation bundle $\mathscr{L}_{FM}$ comes with additional structure: a fusion product. The fusion product distinguishes a class of \emph{fusion-preserving orientations} of $LM$, and Theorem \ref{main} implies 

\begin{corollary}[{{see \cite[Theorem 9]{stolz3}}}]
\label{co:spin}
Let $M$ be a connected, oriented Riemannian manifold. Then, there is a canonical bijection
\begin{equation*}
\bigset{3.9cm}{Equivalence classes of spin structures on $M$} \cong \bigset{3.3cm}{Fusion-preserving orientations of $LM$}\text{.}
\end{equation*}
\end{corollary}

The following example illuminates that it is important to distinguish between orientations and fusion-preserving orientations. Let $M = \mathscr{E}$ be the Enriques surface. Its second Stiefel-Whitney class is non-zero but transgresses to zero. In other words, $\mathscr{E}$ is not spin but $L\mathscr{E}$ is orientable. However, none of the orientations of $L\mathscr{E}$ is fusion-preserving.

\subsection{Second Example: Complex Spin Structures}

We recall that the group $\spinc n$ is the quotient of $\spin n \times \ueins$ by the diagonal $\Z_2$ subgroup, and fits into the central extension
\begin{equation}
\label{spinc}
\alxydim{}{1 \ar[r] & \ueins \ar[r] & \spinc n \ar[r] & \so n \ar[r] & 1\text{,}}
\end{equation}
whose arrows are the  induced by the obvious inclusion and projection, respectively. Accordingly, the Lie algebra $\hat{\mathfrak{g}}$ of $\spinc n$ is a direct sum of the Lie algebras $\mathfrak{g}$ of $\so n$ and $\R$ of $\ueins$. In particular, we have a canonical split of the Lie algebra extension, namely $\sigma(X) := (X,0)$, for which the cocycle $\omega$ and  the map $Z$ we have looked at in Section \ref{sec:lifting} are identically zero.

A \emph{complex spin structure} on an oriented Riemannian manifold  $M$ is  a $\spinc n$-lift $\hat {FM}$ of the frame bundle $FM$. A \emph{spin connection} is a  connection $\hat\eta$ on $\hat {FM}$ compatible with the Levi-Cevita connection on $FM$, and the pair of the two is called a \emph{geometric complex spin structure} on $M$. 

With the above choice of the split $\sigma$, one can take the trivial reduction $r \equiv 0$. The only non-zero quantity is the associated map $r_{\sigma}$ that is used to determine the scalar curvature \erf{scurv} of a spin connection: it is $r_{\sigma}(\hat p, (X,x)) = x$ for $(X,x) \in \mathfrak{g} \oplus \R$ and all $\hat p \in \hat{FM}$. Since $\sigma$ and $r$ are canonically given,  the scalar curvature of a spin connection as well as the principal $\ueins$-bundle $\mathscr{L}_{FM}$ over $LM$ are also independent of any choices.

\begin{remark}
\label{re:scurv}
There is a nice interpretation of the scalar curvature of a spin connection. It is well-known that a complex spin structure $\hat{FM}$ is the same as a principal $\ueins$-bundle $L$ over $M$ together with a $\spinc n$-structure on $FM \times L$, where now $\spinc n$ is viewed as a central extension of $\so n \times S^1$ by $\Z_2$. This correspondence also works in a setup with connections: a spin connection $\hat\eta$ is the same as  a connection on $L$. The curvature of this connection is a 2-form on $M$ and is twice  the scalar curvature of $\hat\eta$. 
\end{remark}

\begin{comment}
The principal $\ueins$-bundle $L$ is given by the formula
\begin{equation*}
L := \hat {FM} \times_{\spinc n} S^1\text{,}
\end{equation*}
where $\spinc n$ acts on $S^1$ by multiplication along the map $\ell : \spinc n \to S^1 : [g,z] \mapsto z^2$. The connection is induced by
\begin{equation*}
\omega := \ell_{*}(\mathrm{pr}_1^{*}\hat\eta) + \mathrm{pr}_2^{*}\theta \in \Omega^1(\hat P \times S^1)\text{.}
\end{equation*}
Notice that $\ell_{*}(X,x) = 2x$. Thus,
\begin{equation*}
\mathrm{scurv}(\hat\eta) = r_{\sigma}(\mathrm{curv}(\hat\eta)) = \frac{1}{2}\ell_{*}(\hat\eta)=\frac{1}{2}\mathrm{curv}(\omega)\text{.}
\end{equation*}
\end{comment}

Returning to the principal $\ueins$-bundle $\mathscr{L}_{FM}$ over $LM$, we see from the alternative construction of Section \ref{sec:alternative} that its total space is
\begin{equation*}
\mathscr{L}_{FM} = (LFM \times \ueins) \;/\; L\so n\text{,}
\end{equation*}
where $L\so n$ acts on $\ueins$ by the holonomy $\mathrm{Hol}_{\nu}$ of the canonical (flat) connection $\nu$ on the underlying $\ueins$-bundle of the central extension \erf{spinc}. Recall that the holonomy of a flat connection is locally constant; in fact one can check that it is the composition of the monodromy $m: L\so n \to \Z_2$ of the spin extension \erf{spin} with the inclusion $\Z_2 \subset \ueins$. 
\begin{comment}
Indeed, for a loop $\tau\in L\so n$ and a lift $\gamma: [0,1] \to \spin n$ we have a commutative diagram
\begin{equation*}
\alxydim{}{[0,1] \ar[d] \ar[r]^-{\gamma} & \spin n \ar[r]^-{\iota} \ar[d]^{\lambda} & \spinc n \ar[d]^{p} \\ S^1 \ar[r]_-{\tau} & \so n \ar@{=}[r] & \so n}\text{.}
\end{equation*}
So we have
\begin{equation*}
\mathrm{Hol}_{\nu}(\tau) = m(\tau) \cdot \exp \left ( \int_{(\iota \circ \gamma)} \nu \right )\text{.}
\end{equation*}
But the integral vanishes:
\begin{equation*}
\int_{\iota \circ \gamma}\nu = \int_{\gamma} \left ( \iota_{*}(\theta_{\spin n}) - \sigma(\lambda_{*}(\theta_{\spin n})) \right ) = \int_{\gamma} 0 = 0\text{.}
\end{equation*}
\end{comment}
This means that our bundle $\mathscr{L}_{FM}$ is  the \emph{extension} of the orientation bundle of $LM$ along the inclusion $\Z_2 \subset \ueins$. We will thus call it the \emph{complex orientation bundle} of $LM$, and call its sections \emph{complex orientations}. 

\begin{comment}
Let us also look at the connection on $\mathscr{L}_{FM}$. Since the curving $C_{\eta}$ from \erf{curvingform} vanishes, this connection is induced from the 1-form $\omega = 0$ on $LFM$. The actual connection 1-form on $\mathscr{L}_{FM}$ is thus simply given by the Maurer-Cartan form of $\ueins$, in particular, it is flat.  
\end{comment}

Summarizing, Theorem \ref{main} implies

\begin{corollary}
\label{co:spinc}
Let $M$ be a connected, oriented Riemannian manifold and $\rho \in \Omega^2(M)$. Then, there is a canonical bijection
\begin{equation*}
\bigset{5.4cm}{Equivalence classes of geometric complex spin structures on $M$ with scalar curvature $\rho$} \cong 
\bigset{4.1cm}{Fusion-preserving complex orientations of $LM$ with curvature $-L\rho$}\text{.}
\end{equation*}
\end{corollary}

\kobib{../../bibliothek/tex}

\end{document}